\documentclass{amsart}
\usepackage{bm}
\usepackage{a4}
\usepackage{amsmath}
\usepackage{amssymb}
\usepackage{amsthm}
\usepackage[colorlinks=true,linkcolor=red]{hyperref}
\usepackage{algorithm}
\usepackage{algpseudocode}
\usepackage{xspace}

\algnewcommand\algorithmicinput{\textbf{Input:}}
\algnewcommand\Input{\item[\algorithmicinput]}
\algnewcommand\algorithmicoutput{\textbf{Output:}}
\algnewcommand\Output{\item[\algorithmicoutput]}

\newlength{\ldprobleft} 
\newlength{\ldprobmid}  
\newlength{\ldprobright}

\newcommand{\dproblem}[3]{
\begin{equation*}
\parbox{\textwidth}{
\begin{tabular}{ @{} p{\ldprobleft} p{\ldprobmid} p{\ldprobright} @{} }
& \multicolumn{2}{l}{\textbf{\uppercase{#1}}} \\
& {\begin{minipage}[t]{\ldprobmid}Input:\vspace{1.5pt}\end{minipage}} & {\begin{minipage}[t]{\ldprobright}#2\vspace{1.5pt}\end{minipage}} \\ 
& {\begin{minipage}[t]{\ldprobmid}Problem:\end{minipage}} & {\begin{minipage}[t]{\ldprobright}#3\end{minipage}} \\
\end{tabular}}
\end{equation*}
}

\newcommand{\lref}[1]{(\ref{#1})}

\newcommand{\N}{\mathbb{N}}
\newcommand{\alg}[1]{#1}
\newcommand{\oneto}[1]{[#1]}
\newcommand{\subuni}[1]{ {\langle #1 \rangle} }

\newcommand{\polyt}{\textnormal{P}}
\newcommand{\np}{\textnormal{NP}}
\newcommand{\pspace}{\textnormal{PSPACE}}

\newcommand{\smp}{\textnormal{SMP}}

\DeclareMathOperator{\id}{id}

\newtheorem{thm}{Theorem}[section]

\newtheorem{lma}[thm]{Lemma}
\newtheorem{clr}[thm]{Corollary}

\theoremstyle{definition}
\newtheorem{prm}[thm]{Problem}
\newtheorem{dfn}[thm]{Definition}
\newtheorem{exm}[thm]{Example}

\theoremstyle{remark}

\DeclareMathOperator{\idX}{\textbf{id}}
\newcommand{\Mid}{\idX}
\newcommand{\zero}{\bm 0}
\newcommand{\one}{\bm 1}
\newcommand{\flip}{\textbf{not}}
\newcommand{\xtoy}[2]{\bm{#1\mapsto#2}}
\newcommand{\ztoz}{\xtoy{0}{0}}
\newcommand{\ztoo}{\xtoy{0}{1}}
\newcommand{\otoz}{\xtoy{1}{0}}
\newcommand{\Cp}{\textbf{cp}\xspace}
\newcommand{\nCp}{\textbf{ncp}\xspace}

\title{The subpower membership problem for semigroups}
\author{Andrei Bulatov}
\address[Andrei Bulatov]{School of Computing Science, Simon Fraser University, Burnaby BC, Canada}
\email{andrei.bulatov@gmail.com}

\author{Marcin Kozik}
\address[Marcin Kozik]{Theoretical Computer Science, Faculty of Mathematics and Computer Science, Jagiellonian University, Poland}
\email{marcin.kozik@uj.edu.pl}

\author{Peter Mayr}
\address[Peter Mayr]{Institute for Algebra, Johannes Kepler University Linz, Austria \& Department of Mathematics, CU Boulder, USA}
\email{peter.mayr@colorado.edu}

\author{Markus Steindl}
\address[Markus Steindl]{Institute for Algebra, Johannes Kepler University Linz, Austria \& Department of Mathematics, CU Boulder, USA}
\email{markus.steindl@colorado.edu}

\date{\today}
\thanks{This material is based upon work supported by the Austrian Science Fund (FWF): P24285,
 the National Science Centre Poland: UMO-2014/13/B/ST6/01812,
 the National Science Foundation under Grant No. DMS 1500254, and
 an NSERC Discovery grant.}
\begin{document}

\begin{abstract}
Fix a finite semigroup $S$ and let $a_1,\ldots,a_k, b$ be tuples in a direct power $S^n$.
The subpower membership problem (\smp{}) asks whether $b$ can be generated by $a_1,\ldots,a_k$.
If $S$ is a finite group, then there is a folklore algorithm that decides this problem in time polynomial in $n k$.
 For semigroups this problem always lies in \pspace. We show that the $\smp{}$ for a 
 full transformation semigroup on $3$ or more letters is actually \pspace-complete,
 while on $2$ letters it is in \polyt.
 For commutative semigroups, we provide a dichotomy result:
 if a commutative semigroup $S$ embeds into a direct product of a Clifford semigroup and a nilpotent
 semigroup, then $\smp(S)$ is in P;
 otherwise it is \np-complete.
\end{abstract}

\maketitle

\section{Introduction}

Deciding membership is a basic problem in computer algebra.
For permutation groups given by generators, it can be solved in polynomial time using Sims' stabilizer
chains~\cite{Furst1980}.
For transformation semigroups, membership is $\pspace$-complete by a result of Kozen~\cite{Ko:LBNP}.

In this paper we study a particular variation of the membership problem that was proposed by
Willard in connection with the study of constraint satisfaction problems (CSP)~\cite{IM:TLAA,Willard2007}.
 Fix a finite algebraic structure $\alg S$ with
finitely many basic operations. Then the \emph{subpower membership problem} (\smp{}) for $\alg S$
is the following decision problem:
\dproblem{SMP(\textit{S})}{
$\{a_1,\ldots,a_k\} \subseteq S^n,b \in S^n$}{
Is $b$ in the subalgebra $\subuni{a_1,\ldots,a_k}$ of $S^n$ generated by\\ $\{a_1,\ldots,a_k\}$?}
For example, for a one-dimensional vector space $S$ over a field $F$, $\smp(S)$ asks whether a 
vector $b \in F^n$ is spanned by vectors $a_1,\ldots,a_k \in F^n$.

Note that  $\smp(S)$ has a positive answer iff there exists a $k$-ary 
term function $t$ on $S$ 
such that $t( a_1,\ldots,a_k ) = b$, that is
\begin{equation} \label{eq:tabi}  
t( a_{1i},\ldots,a_{ki} ) = b_i \quad\text{for all} \quad i\in\{1,\dots,n\}.
\end{equation}
 Hence $\smp(S)$ is equivalent to the following problem: Is the partial operation $t$ that is
 defined on an $n$ element subset of $S^k$ by~\eqref{eq:tabi} the restriction of a term function
 on $\alg S$? 

Note that the input size of $\smp(S)$ is essentially $n(k+1)$.
Since the size of $\subuni{ a_1,\ldots,a_k }$ is limited by $|S|^n$, 
one can enumerate all elements in time exponential in $n$ using a straightforward closure
 algorithm.
This means that $\smp(S)$ is in EXPTIME for each algebra $S$.
Kozik constructed a class of algebras which actually have EXPTIME-complete subpower membership
problems \cite{Kozik2008}.

Still for certain structures the \smp{} might be considerably easier. 
For $S$ a vector space, the \smp{} can be solved by Gaussian elimination in polynomial time. 
For groups the \smp{} is in \polyt{} as well by an adaptation of permutation group algorithms 
\cite{Furst1980,Zweckinger2013}.
Even for certain generalizations of groups and quasigroups the \smp{} can be shown to be in
\polyt{} \cite{Mayr2012}.

In the current paper we start the investigation of algorithms for the \smp{} of finite semigroups and its complexity. 
We will show that the \smp{} for arbitrary semigroups is in $\pspace$ in Theorem~\ref{thm:pspace}
 For the full transformation semigroups $T_n$ on $n$ letters we will prove the following
 in Section~\ref{se:sg}.
\begin{thm}  \label{thm:tn}
 $\smp(T_n)$ is \pspace-complete for all $n\geq 3$, while $\smp(T_2)$ is in~\polyt.
\end{thm}
 This is the first example of a finite algebra with \pspace-complete \smp{}.
 As a consequence we can improve a result of Kozen from~\cite{Ko:LBNP} on the intersection of regular
 languages in Corollary~\ref{co:automata}.

 Moreover
 the following is the smallest semigroup and the first example of an algebra with $\np$-complete
 \smp{}.
\begin{exm}
Let $Z_2^1 := \{ 0,a,1 \}$ denote the 2-element null semigroup adjoined with a $1$,
i.e., $Z_2^1$ has the following multiplication table:
\[ \begin{array}{r|lll}
Z_2^1 & 0 & a & 1 \\
\hline
0 &  0 & 0 & 0 \\
a &  0 & 0 & a \\
1 &  0 & a & 1 \\
\end{array} \]
Then $\smp( Z_2^1 )$ is \np-complete.
$\np$-hardness follows from Lemma \ref{lma:not_clifford_nilpotent_sg_0} by encoding the
exact cover problem. That the problem is in $\np$ for commutative semigroups is proved in
 Lemma \ref{lma:smp_commutative_semigroups}.
\end{exm}

Generalizing from this example we obtain the
 the following dichotomy for commutative semigroups.

\newcommand{\ccsg}{Clifford semigroup}
\begin{thm}
\label{thm:tfae_clifford_nilpotent_commut}
Let $S$ be a finite commutative semigroup. 
Then $\smp(\alg S)$ is in \polyt{} if one of the following equivalent conditions holds:
\begin{enumerate}
\item 
\label{lma:tfae_clifford_nilpotent_list_2}
$S$ is an ideal extension of a \ccsg{} by a nilpotent semigroup;
\item
\label{lma:tfae_clifford_nilpotent_list_3}
the ideal generated by the idempotents of $\alg S$ is a \ccsg{};
\item
\label{lma:tfae_clifford_nilpotent_list_3b}
for every idempotent $e \in S$ and every $a \in S$ where $ea = a$ the element $a$ generates a group;
\item 
\label{lma:tfae_clifford_nilpotent_list_1b}
$S$ embeds into the direct product of a Clifford semigroup and a nilpotent semigroup.
\end{enumerate}
Otherwise $\smp(\alg S)$ is \np-complete.
\end{thm}

Theorem \ref{thm:tfae_clifford_nilpotent_commut} is proved in Section \ref{sec:SMP_commut}.
Our way towards this result starts with describing a polynomial time algorithm for the \smp{} for
Clifford semigroups in Section \ref{sec:clifford_nilpotent}.
In fact in Corollary~\ref{thm:clifford_nilpotent} we will show that $\smp(\alg S)$ is in \polyt{} 
for every (not necessarily commutative) ideal extension of a \ccsg{} by a nilpotent semigroup.

Throughout the rest of the paper, we write $\oneto{n}:=\{1,\dots,n\}$ for $n\in\N$.
Also a tuple $a \in S^n$ is considered as a function 
$a\colon [n] \rightarrow S$. So the $i$-th coordinate of this tuple is denoted by $a(i)$ 
rather than $a_i$.

\section{Full transformation semigroups}  \label{se:sg}

First we give an upper bound on the complexity of the subpower membership problem for arbitrary
 finite semigroups.

\begin{thm}  \label{thm:pspace}
 The \smp{} for any finite semigroup is in \pspace.
\end{thm}

\begin{proof}
Let $S$ be a finite semigroup. We show that 
\begin{equation} \label{eq:NPS}
\smp(S) \text{ is in nondeterministic linear space.}
\end{equation}
To this end, let $A\subseteq S^n,\,b\in S^n$ be an instance of $\smp(S)$.
If $b\in\subuni{A}$, then there exist $a_1,\dots,a_m\in A$ such that $b = a_1\cdots a_m$.

Now we pick the first generator $a_1 \in A$ nondeterministically and start with $c := a_1$.
Pick the next generator $a \in A$ nondeterministically, compute $c := c\cdot a$, and repeat until 
we obtain $c = b$. Clearly all computations can be done in space linear in $n\cdot |A|$.
 This proves~\eqref{eq:NPS}.
By a result of Savitch~\cite{Savitch1970} this implies that $\smp(S)$ is in deterministic quadratic 
space.
\end{proof}


 The first part of Theorem~\ref{thm:tn} follows from the next result since $T_3$ embeds into $T_n$ for all
 $n\geq 3$.

\begin{thm}\label{thm:t3}
  $\smp(T_3)$ is \pspace-complete.
\end{thm}

\begin{proof}
  Kozen~\cite{Ko:LBNP} showed that the following decision problem is \pspace-complete:
  input $n$ and functions $f,f_1,\dotsc,f_m:\oneto{n}\rightarrow\oneto{n}$ and decide whether 
  $f$ can be obtained as a composition%
  \footnote{We will assume that the identity function can be obtained even from an empty set of functions.
  This little twist does not change the complexity of the problem.}
  of $f_i$'s.
  The size of the input for this problem is $(m+1)n\log n$.

  To encode this problem into $\smp(T_3)$ let $T_3$ be the full transformation semigroup of $0,1$, and $\infty$.
  Transformations act on their arguments from the right.
  We identify $g$, an element of $T_3$, with the triple $(0^g,1^g,\infty^g)$
  and name a number of elements of $T_3$: 
  \begin{itemize}
    \item $\zero = (0,0,\infty)$ and $\one = (1,1,\infty)$ are used to encode the functions $\oneto{n}\rightarrow\oneto{n}$;
    \item $\Mid = (0,1,\infty),\,\ztoz = (0,\infty,\infty),\,\ztoo = (1,\infty,\infty)$, and $\otoz = (\infty,0,\infty)$
      are used to model the composition.
  \end{itemize}
  We call an element of $T_3$ {\em bad} if it sends $0$ or $1$ to $\infty$; 
  and we call a tuple of elements {\em bad} if it is bad on at least one position.
  Note that all the named elements
  send $\infty$ to $\infty$.
 So multiplying a bad element on the right by any of the named elements yields a bad element again.

  Let $n$ and $f,f_1,\dots, f_m$ be an input to Kozen's composition problem.
  We will encode it as \smp{} on $n^2+mn$ positions.
  We start with an auxiliary notation.
  Every function $g\colon\oneto{n}\rightarrow \oneto{n}$ can be encoded by a {\em mapping tuple}
 $m_g\in T_3^{n^2+mn}$ as follows:
\[ m_g(x) := \begin{cases} \one & \text{if } x\in\{ 1^g, n+2^g,\dots,(n-1)n+n^g \}, \\
                           \zero & \text{otherwise}. \end{cases} \] 
 Hence the first $n$ positions encode the image of $1$, the next $n$ positions the image of $2$, and so on.
 The final $mn$ positions are used to distinguish mapping tuples from other tuples that we will define shortly.
 Note that mapping tuples are never bad.
    
  We introduce the generators of the subalgebra of $T_3^{n^2+mn}$ gradually.
  The first generator is the mapping tuple $m_1$ for the identity on $\oneto n$.

  Next, for each $f_i$ we add the {\em choice tuple} $c_i$ defined as
\[ c_i(x) := \begin{cases} \Mid & \text{if } x\in\oneto{n^2}, \\
                           \ztoo & \text{if } x\in\{n^2+(i-1)n + 1,\dotsc,n^2+(i-1)n+n\}, \\
                           \ztoz & \text{otherwise}. \end{cases} \] 
  Multiplying the mapping tuple for $g$ on the right by the choice tuple for $f_i$
  corresponds to deciding that $g$ will be composed with $f_i$.

  Finally, for each $f_i$ and $j,k\in\oneto n$ we add the {\em application tuple} $a_{ijk}$ with the semantics
  \begin{equation*}
    \text{apply $f_i$ on coordinate $j$ to $k$.}
  \end{equation*}
 If $k\neq k^{f_i}$, then
\[ a_{ijk}(x) := \begin{cases} \otoz & \text{if } x\in\{(j-1)n+k,n^2+(i-1)n+j\}, \\
                           \ztoo & \text{if } x=(j-1)n+k^{f_i}, \\
                           \Mid & \text{otherwise}. \end{cases} \] 
 If $k = k^{f_i}$, then
\[ a_{ijk}(x) := \begin{cases} \otoz & \text{if } x=n^2+(i-1)n+j, \\
                           \Mid & \text{otherwise}. \end{cases} \] 
 Multiplication by the application tuples computes the composition decided by the choice tuples.
 More precisely, for $g\in T_n$ and $f_i$ we have
\begin{equation} \label{eq:mgfi}
 m_{gf_i} = m_g c_i a_{i11^g}\cdots a_{inn^g}. 
\end{equation}
 Here multiplying $m_g$ by $c_i$ turns the $i$-th block of $n$ positions among the last $nm$ positions 
 of $m_g$ to $\one$. The following multiplication with $a_{i11^g}\cdots a_{inn^g}$ resets these $n$ positions to
 $\zero$ again. At the same time, in the first $n$ positions of $m_gc_i$ the $\one$ gets moved from position $1^g$ to
 $(1^g)^{f_i}$, in the next $n$ positions the $\one$ gets moved from $n+2^g$ to $n+(2^g)^{f_i}$, and so on.
 Hence we obtain the mapping tuple of $gf_i$, and~\eqref{eq:mgfi} is proved.

  It remains to choose an element which will be generated by all these tuples
  iff $f$ is a composition of $f_i$'s. 
  This final element is the mapping tuple for $f$.
 We claim
\begin{equation} \label{eq:fiffmf}
 f\in\langle f_1,\dots, f_m\rangle \text{ iff } m_f\in\langle m_1, c_1,\dots,c_m, a_{111},\dots,a_{mnn}\rangle.
\end{equation}
 The implication from left to right is immediate from our observation~\eqref{eq:mgfi}.
 For the converse we analyze a minimal product of generator tuples which yields $m_f$ and show
 that it essentially follows the pattern from~\eqref{eq:mgfi}.
 Recall that no partial product starting in the leftmost element of the product can be bad.
 In particular the leftmost element itself needs to be $m_1$ -- the only generator which is not bad.
 If $m_1$ occurs anywhere else, then the product could be shortened as any tuple which is
 not bad multiplied by $m_1$ yields $m_1$ again. So we can disregard this case.

 The second element from the left cannot be an application tuple as the $\otoz$ on one of the last $mn$ positions
 would turn the result bad. Thus the only meaningful option is the choice tuple for some function $f_i$.
 Multiplying $m_1$ by $c_i$ turns $n$ positions~(among the last $mn$ positions) of $m_1$ to $\one$.
   
 The third element from the left cannot be a choice tuple:
 a multiplication by a choice tuple produces a bad result unless the last $mn$ positions 
 of the left tuple are all $\zero$. So before any more choice tuples occur in our product, 
 all $n$ $\one$'s in the last $mn$ positions have to be reset to $\zero$. This can only be achieved by 
 multiplying with $n$ application tuples of the form $a_{ijk_j}$ for $j\in\oneto{n}$.
 Focusing on the first $n^2$ positions of $m_1c_i$, we see that necessarily $k_j = j$ for all $j$.
 Hence the first $n+2$ factors of our product are
\[ m_1 c_i a_{i11}\cdots a_{inn} = m_{f_i}. \]  
 Note that the order of the application tuples do not matter.

 Continuing this reasoning with the mapping tuple for $f_i$~(instead of the identity),
 we see that the next $n+1$ factors of our product are some $c_j$ followed by $n$ application tuples
 $a_{j11^{f_i}},\dots,a_{jnn^{f_i}}$. Invoking~\eqref{eq:mgfi} we then get the mapping tuple for $f_i f_j$.
  In the end we get a mapping tuple for $f$ iff 
  $f$ can be obtained as a composition of the $f_i$'s and the identity. This proves~\eqref{eq:fiffmf}.

  The number of tuples we input into \smp{} is $mn^2+m + 2$,
  so the total size of the input is $\mathcal{O}((mn^2+m+2)(n^2+mn))$, that is, polynomial
  with respect to the size of the input of the original problem.
 Thus Kozen's composition problem has a polynomial time reduction to $\smp(T_3)$ and the latter is \pspace-hard
 as well. Together with Theorem~\ref{thm:pspace} this yields the result.
\end{proof}

 Next we show the second part of Theorem~\ref{thm:tn}.
      
\begin{thm}\label{thm:t2}
  $\smp(T_2)$ is in \polyt.
\end{thm}
\begin{proof}
  Let the underlying set of $T_2$ be $\{0,1\}$ and the constants of $T_2$ be denoted by
  $\zero$ and $\one$ and the non-constants by $\Mid$ and $\flip$.
  For a tuple $a\in T_2^n$ the {\em constant part}~(or \Cp) of $a$ is the set of indices $i\in\oneto{n}$ such that
 $a(i)\in T_2$ is a constant,  the {\em non-constant part}~(or \nCp) are the remaining $i$'s.

  Let $a_1,\dotsc,a_k,b\in T_2^n$ be an instance of $\smp(T_2)$.
  Before starting the algorithm we preprocess the input 
  by removing all the $a_i$'s with \Cp not included in \Cp of $b$.
  It is clear that the removed tuples cannot occur in a product that yields $b$.
  Next we call the function $\textbf{SMP}(a_1,\dotsc,a_k,b)$ from Algorithm~\ref{alg:t2}.
  
  \begin{algorithm}
    \caption{ 
    \newline{} Function $\textbf{SMP}(a_1,\dotsc,a_k,b)$ solving $\smp(T_2)$.}
    \label{alg:t2}
    \begin{algorithmic}[1]
      \Input{$a_1,\dotsc,a_k,b\in T_2^n$}
      \Output{Is $b \in \subuni{a_1,\dotsc,a_k}$?}
        
        \State let $a_1,\dotsc,a_\ell$ be the $a_i$'s with empty \Cp\\and $a_{\ell+1},\dotsc,a_k$ with non-empty \Cp
        \If{$b$ has empty \Cp}
        \label{t2:z2}
          \State\Return{$b\in\langle a_1,\dotsc,a_\ell\rangle$} \Comment{instance of $\smp(\mathbb{Z}_2)$}
            \label{t2:z2true}
        \EndIf \label{t2:z2fi}
        \For{$i=\ell+1\dotsc n$}
          \label{t2:mainloop}
          \State \Comment{checks if $a_i$ can be the last element of the product with non-empty \Cp}
          \State let $a'_1,\dotsc,a'_\ell$ be projections of $a_1,\dotsc,a_\ell$ to \Cp of $a_i$
          \State let $b'$~(defined on \Cp of $a_i$) be  $b'(j) = \Mid$ if $a_i(j)=b(j)$ and $b'(j)=\flip$ else
          \If{$b'\in\langle a'_1,\dotsc,a'_\ell\rangle$} \Comment{instance of $\smp(\mathbb{Z}_2)$}
            \label{t2:innerif}
            \State assume $b' = a'_{j_1}\dotsb a'_{j_m}$ for $j_1,\dots,j_m\in\oneto{\ell}$ \label{t2:innerif2}
            \State set $c:=ba_{j_1}\dotsb a_{j_m}$
            \State let $a''_1,\dotsc,a''_k,c''$ be projections of $a_1,\dotsc,a_k,c$ to \nCp of $a_i$
            \State \Return{$\textbf{SMP}(a''_1,\dotsc,a''_k,c'')$}
              \label{t2:recursion}
          \EndIf
        \EndFor
      \State \Return{\textbf{FALSE}}
        \label{t2:false}
    \end{algorithmic}
  \end{algorithm}
  

 We show the correctness of Algorithm~\ref{alg:t2} by induction on the size of $\Cp$ of $b$.
  Note that if $b$ has empty \Cp then, by the preprocessing, 
  each $a_i$ has empty \Cp as well and the problem reduces to SMP over ${\mathbb Z}_2$~%
  (which is solvable in polynomial time by Gaussian elimination).
  This is the essence behind lines~\ref{t2:z2}--\ref{t2:z2fi} of the algorithm. 

  If $b$ has non-empty \Cp, we 
  first assume that $b = a_{j_1}\dotsb a_{j_m}$,
  and let $a_{j_{p}}$ be the last element of the product with non-empty \Cp.
  The suffix $a_{j_{(p+1)}}\dotsb a_{j_m}$ consists of elements of empty \Cp
  which multiply $a_{j_{p}}$, on its \Cp, to $b$.
  This means that the condition on line~\ref{t2:innerif} will be satisfied 
  for some $i$~(maybe with $i=j_{p}$, but maybe with some other $i$). 
  Since $b$ is generated by $a_1,\dots,a_k$ by assumption, then so is $c = ba_{j_1}\dotsb a_{j_m}$~%
  (for any sequence computed in a successful test in line~\ref{t2:innerif}).
  Now $c''$ is just a projection of $c$, and the recursive call in line~\ref{t2:recursion} will
 return the correct answer~\textbf{TRUE} by the induction assumption.

 Next assume that $b$ is not generated by $a_1,\dots,a_k$. Seeking a contradiction we suppose
 that the algorithm returns~\textbf{TRUE}. That is, the recursive call in line~\ref{t2:recursion}, 
 in the loop iteration at some $i$, answers~\textbf{TRUE}. Consequently 
 $b' = a'_{j_1}\dotsb a'_{j_m}$ for some $j_1,\dots,j_m\in\oneto{\ell}$ by line~\ref{t2:innerif2} and
 $c''= a''_{i_1}\dots a''_{i_p}$ for some $i_1,\dots,i_p\in\oneto{k}$ by the induction assumption. We claim that 
 \begin{equation} \label{eq:bai}
    b = a_{i_1}\dotsb a_{i_{p}}a_ia_ia_{j_1}\dotsb a_{j_m}.
 \end{equation}
  Indeed on indices from the \Cp of $a_i$ only the last $m+1$ elements matter and they provide proper values
  by the choice of the sequence $j_1,\dotsc,j_m$ computed by the algorithm.
  For the \nCp of $a_i$ the recursive call  provides $c$. 
  Since $a_ia_i$ is $\Mid$ on \nCp of $a_i$ and $a_{j_1}\dotsb a_{j_m}a_{j_1}\dotsb a_{j_m}$ is a tuple of $\Mid$'s~%
  (since all the tuples in the product have empty \Cp's)
  we obtain $b$ on $\nCp$ of $a_i$ as well. This proves~\eqref{eq:bai} and contradicts our assumption
  that $b$ is not generated by $a_1,\dots,a_k$. Hence the algorithm returns~\textbf{FALSE} in this case. 
  
  The complexity of the algorithm is clearly polynomial: 
  The function $\textbf{SMP}$ works in polynomial time, 
  and the depth of recursion is bounded by $n$ as 
  during each recursive call we loose at least one coordinate.
\end{proof}

 For proving that membership for transformation semigroups is \pspace-complete,
 Kozen first showed that the following decision problem is \pspace-complete~\cite{Ko:LBNP}.
\dproblem{Automata Intersection Problem}{
 deterministic finite state automata $F_1,\dots,F_n$ with common \\ alphabet $\Sigma$
}{
 Is there a word in $\Sigma^*$ that is accepted by all of $F_1,\dots,F_n$?
}
 Using the wellknown connection between automata and transformation semigroups we obtain
 the following stronger version of Kozen's result. 

\begin{clr}  \label{co:automata}
 The Automata Intersection Problem restricted to automata with $3$ states is \pspace-complete.
\end{clr}
 
\begin{proof}
 The Automata Intersection Problem is in $\pspace$ by~\cite{Ko:LBNP}. 
 For \pspace-hardness we reduce $\smp(T_3)$ to our problem.
 Let $T_3$ act on $\{0,1,\infty\}$, and let $a_1,\dots,a_k,b\in T_3^n$ be the input of $\smp(T_3)$.

 For each position $i\in [n]$ we introduce three automata
 $F_i^0,\,F_i^1$, and $F_i^{\infty}$ each with the set of states $\{0,1,\infty\}$.
 These automata are responsible for storing the image of $0,\,1$, and $\infty$, respectively, under the transformation on position $i$. 
 The initial state of $F_i^j$ is $j$, its accepting state $j^{b(i)}$.
 The alphabet of the automata is $\{a_1,\dots,a_k\}$.
 For the automaton $F_i^j$ the letter $a_\ell$ maps the state $x$ to $x^{a_\ell(i)}$.

 Now all the $3n$ automata accept a common word $a_{i_1}\dots a_{i_p}$ over $\{a_1,\dots,a_k\}$ iff
 $j^{a_{i_1}\dots a_{i_p}(i)} = j^{b(i)}$ for all $i\in [n], j\in \{0,1,\infty\}$. The latter is equivalent to 
 $b\in\langle a_1,\dots,a_k\rangle$. Thus $\smp(T_3)$ reduces to the Automata Intersection Problem
 for automata with $3$ states which is then \pspace-hard by Theorem~\ref{thm:t3}.
\end{proof}

\section{Nilpotent semigroups}

\begin{dfn}
A semigroup $\alg S$ is called \emph{$d$-nilpotent} for $d \in \mathbb{N}$ if 
\[ \forall x_1,\ldots,x_d, y_1,\ldots,y_d \in S \colon x_1 \dotsm x_d = y_1 \dotsm y_d. \]
It is called \emph{nilpotent} if it is $d$-nilpotent for some $d \in \mathbb{N}$. 
We let $0 := x_1 \dotsm x_d$ denote the zero element of a $d$-nilpotent semigroup $\alg S$.
\end{dfn}

\begin{dfn}
An \emph{ideal extension} of a semigroup $I$ by a semigroup $Q$ with zero is a semigroup
$S$ such that $I$ is an ideal of $S$ and the Rees quotient semigroup $S/I$ is isomorphic to $Q$.
\end{dfn}

\begin{algorithm}
\caption{ 
\newline{}Reduce $\smp(T)$ to $\smp(S)$ for an ideal extension $T$ of $S$ by $d$-nilpotent $N$.}
\label{alg:nilpotent}
\begin{algorithmic}[1]
\Input $A \subseteq T^n,\, b \in T^n$.
\Output{ Is $b \in \subuni{A}$? }
	\If { $b \not\in S^n$ }
	\label{alg:nilpotent_l01}
		\For { $\ell \in \oneto{d-1}$ }
		\label{alg:nilpotent_l02}
			\For { $a_1,\ldots,a_\ell \in A$ }
			\label{alg:nilpotent_l03}
				\If { $b = a_1 \cdots a_\ell$ }
				\label{alg:nilpotent_l04}
					\State \Return { true } 
					\label{alg:nilpotent_l05}
				\EndIf
			\EndFor
		\EndFor
		\label{alg:nilpotent_l08}
		\State \Return { false }
		\label{alg:nilpotent_l09}
	\Else
                \State $B := \{ a_1 \cdots a_k \in S^n \mid k < 2d, a_1, \dots, a_k\in A \}$
                \label{alg:nilpotent_l11}
		\State \Return{ $b\in \subuni{B}$}
		\label{alg:nilpotent_l12}
\Comment{instance of $\smp(S)$ }
	\EndIf
\end{algorithmic}
\end{algorithm}

\begin{thm} \label{thm:nilpextension}
Let $T$ be an ideal extension of a semigroup $S$ by a $d$-nilpotent semigroup $N$.
Then Algorithm \ref{alg:nilpotent} reduces $\smp(T)$ to $\smp(S)$ in polynomial time.
\end{thm}

\begin{proof}
\emph{Correctness of Algorithm \ref{alg:nilpotent}}.
Let $A\subseteq T^n,\, b\in T^n$ be an instance of $\smp(T)$.

Case $b\not\in S^n$. Since $T/S$ is $d$-nilpotent, a product that is equal to $b$ cannot have more than $d-1$ factors.
Thus Algorithm \ref{alg:nilpotent} verifies in lines 
\ref{alg:nilpotent_l02} to \ref{alg:nilpotent_l08} 
whether there are 
$\ell < d$ and $a_1,\ldots,a_\ell \in A$
such that $b = a_1 \dotsm a_\ell$. 
In line \ref{alg:nilpotent_l05}, Algorithm \ref{alg:nilpotent} returns true if such factors exist.
Otherwise false is returned in line \ref{alg:nilpotent_l09}.

Case $b\in S^n$. Let $B$ be as defined in line \ref{alg:nilpotent_l11}. We claim that 
\begin{equation} \label{eq:nilpotent}
b\in\subuni{A} \text{ iff } b\in\subuni{B}.
\end{equation}
The ``if''-direction is clear. For the converse implication assume $b\in\subuni{A}$.
Then we have $\ell\in\N$ and $a_1,\dots,a_\ell\in A$ such that $b = a_1\cdots a_\ell$.
If $\ell < 2d$, then $b\in B$ and we are done. Assume $\ell \geq 2d$ in the following.
Let $q\in\N$ and $r\in\{0,\dots,d-1\}$ such that $\ell = qd+r$.
For $0\leq j \leq q-2$ define $b_j := a_{jd+1}\cdots a_{jd+d}$. 
Further $b_{q-1} := a_{(q-1)d+1}\dotsm a_{\ell}$.
Since $T/S$ is $d$-nilpotent, any product of $d$ or more elements from $A$ is in $S^n$. 
In particular $b_0,\dots, b_{q-1}$ are in $B$. Since \[ b = b_0 \cdots b_{q-1}, \]
we obtain $b\in\subuni{B}$. Hence~\eqref{eq:nilpotent} is proved.

Since Algorithm \ref{alg:nilpotent} returns $b\in\subuni{B}$ in line \ref{alg:nilpotent_l12},
its correctness follows from~\eqref{eq:nilpotent}. 

\emph{Complexity of Algorithm \ref{alg:nilpotent}}.
In lines \ref{alg:nilpotent_l02} to \ref{alg:nilpotent_l08}, 
the computation of each product $a_1 \dotsm a_\ell$ requires $n(\ell-1)$ multiplications in $S$. 
There are $|A|^\ell$ such products of length $\ell$.
Thus the number of multiplications in $S$ is at most $\sum_{\ell=2}^{d-1} n(\ell-1) |A|^\ell$. 
This expression is bounded by a polynomial of degree $d-1$ in the input size $n(|A|+1)$.

Similarly the size of $B$ and the effort for computing its elements is bounded by a polynomial of
degree $2d-1$ in $n(|A|+1)$.
Hence Algorithm \ref{alg:nilpotent} runs in polynomial time.
\end{proof}
\begin{clr}
The \smp{} for every finite nilpotent semigroup is in $\polyt$.
\end{clr}

\begin{proof}
 Immediate from Theorem~\ref{thm:nilpextension}
\end{proof}

\section{Clifford semigroups}
\label{sec:clifford_nilpotent}

Clifford semigroups are also known as semilattices of groups. 
In this section we show that their \smp{} is in \polyt.
First we state some well-known facts on Clifford semigroups and establish some notation.

\begin{lma}[cf. {\cite[p.~12, Proposition 1.2.3]{Howie1995}}]
In a finite semigroup $S$, each $s \in S$ has an \emph{idempotent power} $s^m$ for some 
$m \in \mathbb{N}$, i.e., $(s^m)^2 = s^m$.
\end{lma}

\begin{dfn}
\label{dfn:clifford}
A semigroup $\alg S$ is \emph{completely regular} if every $s \in S$ is contained in a subsemigroup of $\alg S$ which is a also a group.
A semigroup $\alg S$ is a \emph{Clifford semigroup} if it is completely regular and its idempotents are central. The latter condition may be expressed by
\begin{equation*}
\label{formula:idempotents_commute}
\forall e,s \in S \colon ( e^2 = e \Rightarrow es = se ) \text{.} 
\end{equation*}
\end{dfn}

\begin{dfn}
\label{dfn:strong_sl_of_grps}
Let $\langle I, \wedge \rangle$ be a semilattice.
For $i \in I$ let $\langle G_i, \cdot \rangle$ be a group.
For $i, j, k \in I$ with $i \geq j \geq k$ let 
$\phi_{i,j} \colon \alg G_i \rightarrow \alg G_j$
be group homomorphisms such that 
$\phi_{j,k} \circ \phi_{i,j} = \phi_{i,k}$
and $\phi_{i,i} = \id_{G_i}$.
Let $S := \dot{\bigcup}_{i \in I} G_i$, and 
\[ \text{for} \quad x \in G_i ,\, y \in G_j \quad \text{let} \quad
x*y := \phi_{i,i \wedge j}(x) \cdot \phi_{j,i \wedge j}(y). \]
Then we call $\langle S, * \rangle$ a \emph{strong semilattice of groups}.
\end{dfn}

\begin{thm}[Clifford, cf. {\cite[p.~106--107, Theorem 4.2.1]{Howie1995}} ] 
\label{thm:clifford_decomposition}
A semigroup is a strong semilattice of groups iff it is a Clifford semigroup.
\end{thm}
Note that the operation $*$ extends the multiplication of $G_i$ for each $i \in I$.
It is easy to see that $\{ G_i \mid i \in I \}$ are precisely the maximal subgroups of $S$.
Moreover, each Clifford semigroup inherits a preorder $\leq$ from the underlying semilattice.

\begin{dfn}
Let $\alg S$ be a Clifford semigroup constructed from a semilattice $\alg I$ and disjoint groups $\alg G_i$ for $i \in I$ as in Definition \ref{dfn:strong_sl_of_grps}.
For $x, y \in S$ define
\[ x \leq y \quad\text{if}\quad \exists i,j \in I \colon i \leq j, x \in G_i, y \in G_j. \]
\end{dfn}

\begin{lma}
\label{prp:clifford_preorder}
Let $\alg S$ be a Clifford semigroup and $x,y,z \in S$. Then
\begin{enumerate}
	\item \label{list:clifford_preorder_1}
	$x \leq yz$ iff $x \leq y$ and $x \leq z$,
	\item \label{list:clifford_preorder_2}
	$xyz \leq y$, and
	\item \label{list:clifford_preorder_3}
	$x \leq y$ and $y \leq x$ iff $x$ and $y$ are in the same maximal subgroup of $S$.
\end{enumerate}
\end{lma}
\begin{proof}Straightforward.\end{proof}

The following mapping will help us solve the \smp{} for Clifford semigroups.
\begin{dfn}
Let $\alg S$ be a finite Clifford semigroup constructed from a semilattice $\alg I$ and disjoint groups $\alg G_i$ for $i \in I$ as in Definition \ref{dfn:strong_sl_of_grps}.
Let 
\[ \gamma \colon 
S \rightarrow \prod_{i \in I} G_i \quad\text{such that}\quad \gamma(s)(i) := 
\begin{cases}
s & \text{if } s \in G_i \text{,} \\
1_{G_i} & \text{otherwise}
\end{cases} \]
for $s \in S$ and $i \in I$.
\end{dfn}

Here $\prod$ denotes the direct product and $1_{G_i}$ the identity of the group $G_i$ for $i \in I$.
Note that the mapping $\gamma$ is not necessarily a homomorphism.

\begin{algorithm}
\caption{
\newline{}For a Clifford semigroup $S = \dot\bigcup_{i \in I} G_i$,
reduce $\smp( S )$ to $\smp( \prod_{i \in I} G_i )$.
}
\label{alg:clifford}
\begin{algorithmic}[1]
\Input $A \subseteq S^n,\, b \in S^n$.
\Output{ True if $b \in \subuni{ A }$, false otherwise. }
		\State Set $\{ a_1,\ldots,a_k \} := 
						\{ a \in A \mid \forall i \in \oneto{n} \colon a(i) \geq b(i) \}$
		\label{alg:clifford:a1ak}
		\State Set $e$ to the idempotent power of $b$.
		\label{alg:clifford:l2}
		\If{ $\exists i \in \oneto{n} \colon e(i) \notin \subuni{ a_1(i),\ldots,a_k(i) }$}
		\label{alg:clifford:exist_e}	
			\State \Return{ false }
			\label{alg:clifford:ret1}	
		\EndIf
			\State \Return{ $\gamma(b) \in \subuni{ \gamma(a_1 e),\ldots,\gamma(a_k e) }$ }
		\label{alg:clifford:l06}
		\Comment{instance of $\smp( \prod_{i \in I} \alg G_i )$ }
\end{algorithmic}
\end{algorithm}

\begin{thm}
\label{thm:smp_clifford}
Let $\alg S$ be a finite Clifford semigroup with maximal subgroups $G_i$ for $i \in I$.
Then Algorithm \ref{alg:clifford} reduces $\smp(\alg S)$ to $\smp(\prod_{i \in I} G_i)$ in 
polynomial time. The latter is the \smp{} of a group.
\end{thm}

\begin{proof}
\emph{Correctness of Algorithm \ref{alg:clifford}}.
Assume $\alg S = \subuni{ \dot{\bigcup}_{i \in I} G_i , \cdot }$ as in Definition \ref{dfn:strong_sl_of_grps}.
Fix an instance $A \subseteq S^n,\, b \in S^n$ of $\smp(\alg S)$.
Let $a_1,\ldots,a_k$ be as defined in line \ref{alg:clifford:a1ak} 
of Algorithm \ref{alg:clifford}.

First we claim that 
\begin{equation}
\label{eq:claim_alg_cliff} 
b \in \subuni{ A }
\quad \text{iff} \quad 
b \in \subuni{ a_1,\ldots,a_k }.
\end{equation}
To this end, assume that $b = c_1 \dotsm c_m$ for $c_1,\ldots,c_m \in A$.
Fix $j \in \oneto{m}$.
Lemma \ref{prp:clifford_preorder}\lref{list:clifford_preorder_1} implies that 
$b(i) \leq c_j(i)$ for all $i \in \oneto{n}$. 
Thus $c_j \in \{ a_1,\ldots,a_k \}$.
Since $j$ was arbitrary, we have $c_1,\ldots,c_m \in \{ a_1,\ldots,a_k \}$
and \eqref{eq:claim_alg_cliff} follows.

Let $e$ be the idempotent power of $b$.
If the condition in line \ref{alg:clifford:exist_e} of Algorithm \ref{alg:clifford} is fulfilled,
then neither $e$ nor $b$ are in $\subuni{ a_1,\ldots,a_k }$. 
In this case false is returned in line \ref{alg:clifford:ret1}.
Now assume the condition in line \ref{alg:clifford:exist_e} is violated, i.e., 
\begin{equation*}
	\forall i \in \oneto{n} \colon e(i) \in \subuni{ a_1(i),\ldots,a_k(i) }. 
\end{equation*}
We claim that
\begin{equation}
\label{thm:smp_clifford_formula_0}
e \in \subuni{ a_1,\ldots,a_k }.
\end{equation}
For each $i \in \oneto{n}$ let $d_i \in \subuni{ a_1,\ldots,a_k }$ such that $d_i(i) = e(i)$. 
Further let $f$ be the idempotent power of $d_1 \dotsm d_n$.
We show $f = e$. 
Fix $i \in \oneto{n}$.
Since $d_i(i) = e(i)$, we have $f(i) \leq e(i)$ by Lemma \ref{prp:clifford_preorder}\lref{list:clifford_preorder_2}.
On the other hand, $e(i) \leq b(i) \leq a_j(i)$ for all $j \leq k$.
Hence $e(i) \leq f(i)$ by multiple applications of Lemma \ref{prp:clifford_preorder}\lref{list:clifford_preorder_1}.
Thus $f(i)$ and $e(i)$ are idempotent and are in the same group by Lemma \ref{prp:clifford_preorder}\lref{list:clifford_preorder_3}.
So $e(i) = f(i)$.
This yields $e = f$ and thus
\eqref{thm:smp_clifford_formula_0} holds.

Next we show 
\begin{equation}
\label{thm:smp_clifford_formula_1}
b \in \subuni{ a_1,\ldots,a_k }
\quad \text{iff} \quad 
b \in \subuni{ a_1e,\ldots,a_ke }.
\end{equation}
If $b = c_1 \dotsm c_m$ for $c_1,\ldots,c_m \in \{ a_1,\ldots,a_k \}$, then 
$b = be = c_1 \dotsm c_m e = (c_1e) \dotsm (c_me)$ 
since idempotents are central in Clifford semigroups. 
This proves \eqref{thm:smp_clifford_formula_1}.

Next we claim that 
\begin{equation}\label{thm:smp_clifford_formula_2}
b \in \subuni{a_1e,\ldots,a_ke}
\quad \text{iff} \quad 
\gamma(b) \in \subuni{\gamma(a_1e),\ldots,\gamma(a_ke)}.
\end{equation}
Fix $i \in \oneto{n}$.
By Lemma \ref{prp:clifford_preorder}\lref{list:clifford_preorder_3} the elements $a_1e(i),\ldots,a_ke(i)$, and $b(i)$ all lie in the same group, say $\alg G_l$. 
Note that $\gamma|_{G_l} \colon \alg G_l \rightarrow \prod_{i \in I} \alg G_i$ is a semigroup monomorphism.
This means that the componentwise application of $\gamma$ to $\subuni{ a_1e,\ldots,a_ke, b }$, namely
\[ \gamma |_\subuni{ a_1e,\ldots,a_ke, b } \colon \subuni{ a_1e,\ldots,a_ke, b } 
\rightarrow ( \prod_{i \in I} \alg G_i )^n, \] 
is also a semigroup monomorphism. 
This implies \eqref{thm:smp_clifford_formula_2}.

In line \ref{alg:clifford:l06}, the question whether 
$\gamma(b) \in \subuni{\gamma(a_1e),\ldots,\gamma(a_ke)}$ is an instance of 
$\smp( \prod_{i \in I} \alg G_i )$, which is the \smp{} of a group.
By \eqref{eq:claim_alg_cliff}, \eqref{thm:smp_clifford_formula_1}, and 
\eqref{thm:smp_clifford_formula_2}, Algorithm \ref{alg:clifford} returns true iff 
$b \in \subuni{ A }$.

\emph{Complexity of Algorithm \ref{alg:clifford}}.
Line \ref{alg:clifford:a1ak} requires at most $\mathcal{O}( n|A| )$ calls of the relation $\leq$. 
For line \ref{alg:clifford:l2}, 
let $(s_1,\ldots,s_{|S|})$ be a list of the elements of $S$ and 
let $v \in \mathbb{N}$ minimal such that $(s_1,\ldots,s_{|S|})^v$ is idempotent.
Then $e = b^v$. 
Since $v$ only depends on $\alg S$ but not on $n$ or $|A|$,
computing $e$ takes $\mathcal{O}( n )$ steps.
Line \ref{alg:clifford:exist_e} requires $\mathcal{O}( n|A| )$ steps.
Altogether the time complexity of Algorithm \ref{alg:clifford} is $\mathcal{O}( n|A| )$.
\end{proof}

\begin{clr}
\label{clr:clifford}
The \smp{} for finite Clifford semigroups is in \polyt{}.
\end{clr}
\begin{proof}
Let $S$ be a finite Clifford semigroup.
Fix an instance $A \subseteq S^n,\, b \in S^n$ of $\smp(\alg S)$.
Algorithm \ref{alg:clifford} converts this instance into one of the \smp{} of a group with maximal size of $|S|^{|S|}$ in $\mathcal{O}( n|A| )$ time.
Both instances have input size $n(|A| + 1)$.
The latter can be solved by Willard's modification \cite{Willard2007} of the concept of strong generators, known from the permutation group membership problem \cite{Furst1980}.
This requires $\mathcal{O}( n^3 + n|A| )$ time according to \cite[p. 53, Theorem 3.4]{Zweckinger2013}.
Hence $\smp(\alg S)$ is decidable in $\mathcal{O}( n^3 + n|A| )$ time.
\end{proof}

\begin{clr}
\label{thm:clifford_nilpotent}
Let $S$ be a finite ideal extension of a Clifford semigroup by a nilpotent semigroup.
Then $\smp(S)$ is in \polyt{}.
\end{clr}

\begin{proof}
By Theorem \ref{thm:nilpextension} and Corollary \ref{clr:clifford}.
\end{proof}

In the next lemma we give some conditions equivalent to the fact that a semigroup is
an ideal extension of a Clifford semigroup by a nilpotent semigroup.

\begin{lma}
\label{lma:tfae_clifford_nilpotent}
Let $\alg S$ be a finite semigroup. 
Then the following are equivalent:
\begin{enumerate}
	\item \label{lma:tfae_clifford_nilpotent_list_2}
	$S$ is an ideal extension of a Clifford semigroup $C$ by a nilpotent semigroup $N$;
	\item \label{lma:tfae_clifford_nilpotent_list_3}
	the ideal $I$ generated by the idempotents of $\alg S$ is a Clifford semigroup;
	\item \label{lma:tfae_clifford_nilpotent_list_3b}
	all idempotents in $S$ are central, and for every idempotent $e \in S$ and every $a \in S$ where $ea = a$ the element $a$ generates a group;
	\item \label{lma:tfae_clifford_nilpotent_list_1b}
	$S$ embeds into the direct product of a Clifford semigroup $\alg C$ and a nilpotent semigroup $N$.
\end{enumerate}
\end{lma}

\begin{proof}
$(\ref{lma:tfae_clifford_nilpotent_list_2}) \Rightarrow 
(\ref{lma:tfae_clifford_nilpotent_list_3})$: 
We show $I = C$. Since $S \setminus C$ cannot contain idempotent elements, all idempotents are in the ideal $C$. Thus we have $I \subseteq C$. Now let $c \in C$. Let $e \in I$ be the idempotent power of $c$. Then $c = ce \in I$. So $C \subseteq I$.

$(\ref{lma:tfae_clifford_nilpotent_list_3}) \Rightarrow (\ref{lma:tfae_clifford_nilpotent_list_3b})$: 
First we claim that all idempotents are central in $S$.
To this end, let $e \in S$ be idempotent and $a \in S$.
Then 
\begin{align*}
ae 
&= (ae)e \\
&= e(ae) \qquad \text{since $e,ae \in I$ and $e$ is central in $I$,} \\
&= (ea)e \\
&= e(ea) \qquad \text{since $e,ea \in I$ and $e$ is central in $I$,} \\
&= ea \text{.}
\end{align*}
Next assume that $ea = a$. Since $ea \in I$, we have that $\subuni{ a } = \subuni{ ea }$ is a group.

$(\ref{lma:tfae_clifford_nilpotent_list_3b}) \Rightarrow 
(\ref{lma:tfae_clifford_nilpotent_list_1b})$: 
Let $k \in \mathbb{N}$ such that $x^k$ is idempotent for each $x \in S$. 
For $x \in S$ and an idempotent $e \in S$ we have
\begin{equation}
\label{eq:thm:tfae_cn_eq1}
ex = (ex)^{k+1} = ex^{k+1}
\end{equation}
since $\subuni{ ex }$ is a group and idempotents are central.
We claim that 
\begin{equation}
\label{eq:thm:tfae_cn_eq1.5}
\alpha \colon S \rightarrow S,\, x \mapsto x^{k+1}
\quad
\text{is a homomorphism with} 
\quad
\alpha^2 = \alpha \text{.}
\end{equation}
For $x,y \in S$,
\renewcommand{\arraystretch}{1.4}
\begin{equation*}
\begin{array}{rll}
\label{eq:thm:tfae_cn_eq2}
(xy)^{k+1} \!\!\!\!\!
&= (xy)^kxy \\
&= (xy)^kx^{k+1}y &\quad \text{by \eqref{eq:thm:tfae_cn_eq1} since $(xy)^k$ is idempotent,} \\
&= (xy)^kx^{k+1}y^{k+1} &\quad \text{by \eqref{eq:thm:tfae_cn_eq1} since $x^k$ is idempotent,} \\
&= (xy)^{k+1}x^ky^k &\quad \text{since $x^k, y^k$ are central,} \\
&= xyx^ky^k &\quad \text{by \eqref{eq:thm:tfae_cn_eq1} since $x^k$ is idempotent,} \\
&= x^{k+1}y^{k+1} &\quad \text{since $x^k, y^k$ are central.} \\
\end{array}
\end{equation*}
Also,
\begin{equation*}
\label{eq:thm:tfae_cn_eq3}
(x^{k+1})^{k+1} = x^{k^2+2k+1} = x^{k+1} \text{.}
\end{equation*}
This proves \eqref{eq:thm:tfae_cn_eq1.5}.
Let $C := \alpha( S )$.
We claim that $C$ is an ideal.
For $x, y \in S \cup \{ 1 \}$ and $z^{k+1} \in C$, 
\begin{equation*}
\begin{array}{rll}
\label{eq:thm:tfae_cn_eq4}
xz^{k+1}y \!\!\!\!\!
&= xzyz^k &\quad \text{since $z^k$ is central,} \\
&= (xzy)^{k+1}z^k &\quad \text{by \eqref{eq:thm:tfae_cn_eq1},} \\
&=( xz^{k+1}y )^{k+1} &\quad \text{since $z^k$ is central and idempotent,} \\
&\in C \text{.}
\end{array}
\end{equation*}
\renewcommand{\arraystretch}{1.0}
Now consider the Rees quotient $N := S/C$. 
We claim that
\begin{equation}
\label{eq:thm:tfae_cn_eq5}
\text{$N$ is $|N|$-nilpotent.}
\end{equation}
Let $n_1,\ldots,n_{|N|} \in S$. First assume
\begin{equation}\label{eq:N-N-nilpotent}
\exists i,j \in \{ 1,\ldots,|N| \},\, i < j \colon n_1 \cdots n_i = n_1 \cdots n_j \text{.}
\end{equation}
Then $n_{i+1} \cdots n_j$ is a right identity of $n_1 \cdots n_i$. Thus
\[ n_1 \cdots n_i = n_1 \cdots n_i (n_{i+1} \cdots n_j)^{k+1} \in C \]
since $C$ is an ideal. So $n_1 \cdots n_{|N|} \in C$. 

If \eqref{eq:N-N-nilpotent} does not hold, then 
$n_1, n_1 n_2,\ldots,n_1 \cdots n_{|N|}$ 
are $|N|$ distinct elements and at least one of them is in $C$.
Again $n_1 \cdots n_{|N|} \in C$ by the ideal property of $C$.
This proves \eqref{eq:thm:tfae_cn_eq5}.
Now let
\[ \beta \colon S \rightarrow C \times N,\, s \mapsto ( \alpha(s), s/C ). \]
Apparently $\beta$ is a homomorphism.
It remains to prove that $\beta$ is injective.
Assume $\beta(x) = \beta(y)$ for $x,y \in S$.
If $x \notin C$, then also $y \notin C$. 
Now $x/C = y/C$ implies $x = y$.
Assume $x \in C$. 
Then $x = \alpha(x) = \alpha(y) = y$ since $\alpha^2 = \alpha$.
We proved item (\ref{lma:tfae_clifford_nilpotent_list_1b}) of Lemma \ref{lma:tfae_clifford_nilpotent}.

$(\ref{lma:tfae_clifford_nilpotent_list_1b}) \Rightarrow
 (\ref{lma:tfae_clifford_nilpotent_list_2})$:
Assume $S \leq C \times N$.
Then $J := S \cap ( C \times \{ 0 \} )$ is an ideal of $S$.
At the same time $J$ is a subsemigroup of a Clifford semigroup.
By Definition \ref{dfn:clifford} also $J$ is a Clifford semigroup.
It is easy to see that the Rees quotient $N_1 := S / J$ is nilpotent.
Thus $S$ is an ideal extension of the Clifford semigroup $J$ by the nilpotent semigroup $N_1$.
\end{proof}

\section{Commutative semigroups}
\label{sec:SMP_commut}

The main result of Section \ref{sec:clifford_nilpotent} was that 
ideal extensions of Clifford semigroups by nilpotent semigroups
have the \smp{} in \polyt. 
In this section we show that if a commutative semigroup does not have this property, 
then its \smp{} is \np-complete. 
This will complete the proof of our dichotomy result, Theorem \ref{thm:tfae_clifford_nilpotent_commut}.

First we give an upper bound on the complexity of the \smp{} for commutative semigroups.
\begin{lma}
\label{lma:smp_commutative_semigroups}
The \smp{} for a finite commutative semigroup is in \np.
\end{lma}

\begin{proof}
Let $\{a_1,\ldots,a_k\}\subseteq S^n,\,b \in S^n$ be an instance of $\smp(\alg S)$. 
Let $x := ( s_1,\ldots,s_{|S|} )$ be a list of all elements of $S$, and $r := | \subuni{x} |$. 
Now $\subuni{x} = \{ x^1,\ldots,x^r \}$, and 
for each $\ell \in \mathbb{N}$ there is some $m \in \oneto{r}$ such that $x^\ell = x^m$.
Since $x$ contains all elements of $S$, we have
\[ \forall y \in S^n \, \forall \ell \in \mathbb{N} \, 
\exists m \in \oneto{r} \colon y^\ell = y^m. \]
If $b \in \subuni{ a_1,\ldots,a_k }$, then there is a witness $(\ell_1,\ldots,\ell_k) \in \{ 0,\ldots,r \}^k$ such that $b = {a_1}^{\ell_1} \dotsm {a_k}^{\ell_k}$. 
The size of this witness is $\mathcal{O}( k \log(r) )$. 
Note that $r$ depends only on $\alg S$ and not on the input size $n(k+1)$.
Given $\ell_1,\ldots,\ell_k$ we can verify $b = {a_1}^{\ell_1} \dotsm {a_k}^{\ell_k}$ in time polynomial in $n(k+1)$. Hence $\smp(\alg S)$ is in \np{}.
\end{proof}
\begin{lma}
\label{lma:not_clifford_nilpotent_sg_0}
Let $\alg S$ be a finite semigroup, $e \in S$ be idempotent, and $a \in S$.
Assume that $ea = ae = a$ and $\subuni{ a }$ is not a group. 
Then $\smp(S)$ is \np-hard.
\end{lma}
\begin{proof}
We reduce EXACT COVER to $\smp( \alg S )$.
The former is one of Karp's 21 \np-complete problems \cite{Karp1972}.
\dproblem{Exact Cover}{
$n \in \mathbb{N}$, sets $C_1,\ldots,C_k \subseteq \oneto{n}$ 
}{
Are there disjoint sets $D_1,\ldots,D_m \in \{ C_1,\ldots,C_k \}$ such that 
$\bigcup_{i=1}^{m} D_i = \oneto{n}$?
}
Fix an instance $n, C_1,\ldots,C_k$ of EXACT COVER. 
Now we define characteristic functions $c_1,\ldots,c_k, b \in S^n$ for $C_1,\ldots,C_k, \oneto{n}$, respectively.
For $j \in\oneto{k},\, i \in \oneto{n}$, let
\begin{equation*}
b(i) := a
\quad
\text{and}
\quad
c_j(i) := 
\begin{cases}
   a & \text{if } i \in C_j \text{,} \\
   e & \text{otherwise.}
\end{cases}
\end{equation*}
Now let $\{c_1,\ldots,c_k\} \subseteq S^n,\,b \in S^n$ be an instance of $\smp(\alg S)$.
We claim that
\begin{align*}
&b \in \subuni{ c_1,\ldots,c_k } \quad\text{iff}\quad
\exists \text{ disjoint } D_1,\ldots,D_m \in \{ C_1,\ldots,C_k \} \colon 
\bigcup_{i=1}^{m} D_i = \oneto{n}.
\end{align*}

"$\Rightarrow$": 
Let $d_1,\ldots,d_m \in \{ c_1,\ldots,c_k \}$ such that $b = d_1 \cdots d_m$.
Let $D_1,\ldots,D_m$ be the sets corresponding to $d_1,\ldots,d_m$, respectively.
Then $\bigcup_{i=1}^{m} D_i = \oneto{n}$. The union is disjoint since $a \notin \{ a^2, a^3,\ldots \}$.

"$\Leftarrow$": 
Fix $D_1,\ldots,D_m$ whose disjoint union is $\oneto{n}$.
Let $d_1,\ldots,d_m \in \{ c_1,\ldots,c_k \}$ be the characteristic functions of $D_1,\ldots,D_m$, respectively. 
Then $b = d_1 \dotsm d_m$.
\end{proof}

\begin{clr}\label{clr:not_clifford_nilpotent_sg}
Let $\alg S$ be a finite commutative semigroup that does not fulfill one of the equivalent 
conditions of Lemma \ref{lma:tfae_clifford_nilpotent}. Then $\smp(\alg S)$ is \np-hard.
\end{clr}

\begin{proof}
The semigroup $S$ violates condition (\ref{lma:tfae_clifford_nilpotent_list_3b}) of Lemma \ref{lma:tfae_clifford_nilpotent}.
Since the idempotents are central in $S$,
there are $e \in S$ idempotent and $a \in S$ such that $ea = ae = a$ and $\subuni a$ is not a group.
Now the result follows from Lemma \ref{lma:not_clifford_nilpotent_sg_0}.
\end{proof}

Now we are ready to prove our dichotomy result for commutative semigroups.

\begin{proof}[Proof of Theorem \ref{thm:tfae_clifford_nilpotent_commut}]
The conditions in 
Theorem \ref{thm:tfae_clifford_nilpotent_commut} are the ones from 
Lemma \ref{lma:tfae_clifford_nilpotent} adapted to the commutative case.
Thus they are equivalent. 
If one of them is fulfilled, then $\smp(\alg S)$ is in \polyt{} by Corollary~\ref{thm:clifford_nilpotent}.

Now assume the conditions are violated.
Then $\smp(\alg S)$ is \np-complete by 
Lemma \ref{lma:smp_commutative_semigroups} and Corollary \ref{clr:not_clifford_nilpotent_sg}.
\end{proof}

\section{Conclusion}
\label{sec:conclusion}

We showed that the \smp{} for finite semigroups is always in $\pspace$ and provided examples
 of semigroups $S$ for which $\smp(S)$ is in \polyt, \np-complete, \pspace-complete, respectively.
For the \smp{} of commutative semigroups we obtained a dichotomy between the $\np$-complete
and polynomial time solvable cases. Further we showed that the \smp{} for finite ideal extensions
 of a Clifford semigroup by a nilpotent semigroup is in \polyt. For non-commutative semigroups
 there are several open problems.

\begin{prm}
 Is the \smp{} for every finite semigroup either in \polyt, \np-complete, or \pspace-complete?
\end{prm}

Bands (idempotent semigroups) are well-studied. Still we do not know the following:

\begin{prm}
What is the complexity of the \smp{} for finite bands?%
\footnote{While this paper was under review, Markus Steindl showed that \smp{} for any finite band is either in \polyt ~or $\np$-complete \cite{smpbands}.}
More generally, what is the complexity in case of completely regular semigroups?
\end{prm}


\end{document}